\documentclass[reqno]{amsart}

\usepackage{tikz}
\usepackage{hyperref}

\newcommand*{\mailto}[1]{\href{mailto:#1}{\nolinkurl{#1}}}

\newcommand{\doi}[1]{\href{http://dx.doi.org/#1}{doi: #1}}

\newtheorem{theorem}{Theorem}[section]
\newtheorem{definition}[theorem]{Definition}
\newtheorem{lemma}[theorem]{Lemma}
\newtheorem{proposition}[theorem]{Proposition}
\newtheorem{corollary}[theorem]{Corollary}

\newtheorem{remark}[theorem]{Remark}

\newcommand{\R}{{\mathbb R}}

\newcommand{\Z}{{\mathbb Z}}
\newcommand{\C}{{\mathbb C}}

\newcommand{\E}{\mathrm{e}}
\newcommand{\I}{\mathrm{i}}

\newcommand{\supp}{\mathrm{supp}}
\newcommand{\tr}{\mathrm{tr}}
\newcommand{\im}{\mathrm{Im}}

\newcommand{\per}{{\mathrm{per}}}
\newcommand{\loc}{{\mathrm{loc}}}
\newcommand{\cc}{{\mathrm{c}}}

\newcommand{\cprod}{\prod^{\leftrightarrow}}
\newcommand{\csum}{\sum^{\leftrightarrow}}

\newcommand{\ledot}{\,\cdot\,}
\newcommand{\redot}{\cdot\,}

\newcommand{\qd}{{[1]}}
\newcommand{\dip}{\upsilon}
\newcommand{\period}{\ell}
\newcommand{\inds}{\mathcal{I}}
\newcommand{\Dper}{\mathcal{D}_{\per}}

\numberwithin{equation}{section}

\begin{document}

\title[Trace formulas and continuous dependence]{Trace formulas and continuous dependence of spectra for~the~periodic~conservative Camassa--Holm flow}
 
\author[J.\ Eckhardt]{Jonathan Eckhardt}
\address{Department of Mathematical Sciences\\ Loughborough University\\ Loughborough LE11 3TU \\ UK\\  and  Faculty of Mathematics\\ University of Vienna\\ Oskar-Morgenstern-Platz 1\\ 1090 Wien\\ Austria}
\email{\mailto{J.Eckhardt@lboro.ac.uk};\ \mailto{jonathan.eckhardt@univie.ac.at}}
\urladdr{\url{http://homepage.univie.ac.at/jonathan.eckhardt/}}

\author[A.\ Kostenko]{Aleksey Kostenko}
\address{Faculty of Mathematics and Physics\\ University of Ljubljana\\ Jadranska 19\\ 1000 Ljubljana\\ Slovenia\\ and Faculty of Mathematics\\ University of Vienna\\ Oskar-Morgenstern-Platz 1\\ 1090 Wien\\ Austria}
\email{\mailto{Aleksey.Kostenko@fmf.uni-lj.si};\ \mailto{Oleksiy.Kostenko@univie.ac.at}}
\urladdr{\url{http://www.mat.univie.ac.at/~kostenko/}}

\author[N.\ Nicolussi]{Noema Nicolussi}
\address{Faculty of Mathematics\\ University of Vienna\\
Oskar-Morgenstern-Platz 1\\ 1090 Vienna\\ Austria}
\email{\mailto{noema.nicolussi@univie.ac.at}}

\thanks{\href{http://dx.doi.org/10.1016/j.jde.2019.09.048}{J.\ Differential Equations {\bf 268} (2020), no.~6, 3016--3034}}
\thanks{{\it Research supported by the Austrian Science Fund (FWF) under Grants No.~P29299 and P28807 as well as by the Slovenian Research Agency (ARRS) under Grant No.~J1-9104}}


\begin{abstract}
This article is concerned with the isospectral problem 
\[
-f'' + \frac{1}{4} f = z\, \omega f + z^2 \dip f
\]
for the periodic conservative Camassa--Holm flow, where $\omega$ is a periodic real distribution in $H^{-1}_{\loc}(\R)$ and $\dip$ is a periodic non-negative Borel measure on $\R$. 
We develop basic Floquet theory for this problem, derive trace formulas for the associated spectra and establish continuous dependence of these spectra on the coefficients with respect to a weak$^\ast$ topology.  
\end{abstract}

\maketitle

\section{Introduction}
 
 In this article, we are interested in direct spectral theory for the isospectral problem of the conservative Camassa--Holm flow in the periodic setting. 
 The Camassa--Holm equation 
  \begin{equation}\label{eqnCH}
   u_{t} -u_{xxt}  = 2u_x u_{xx} - 3uu_x + u u_{xxx} 
  \end{equation}
  has become one of the most extensively studied partial differential equations during the last few decades. 
  Due to the vast amount of literature devoted to various aspects of this equation, we only refer to a brief selection of articles \cite{besasz02, besasz05, brco07, caho93, coes98, comc99, como00, geho08, hora07, xizh00} containing further information. 
 For an arbitrary but fixed period $\period>0$, the natural phase space $\Dper$ for the periodic conservative Camassa--Holm flow is defined as follows; compare \cite[Definition~2.4]{grhora13}. 
 
\begin{definition}
The set $\Dper$ consists of all pairs $(u,\mu)$ such that $u$ is an $\period$-periodic, real-valued function in $H^1_{\loc}(\R)$ and $\mu$ is an $\period$-periodic, non-negative Borel measure on $\R$ with
\begin{align}\label{eqnmuac}
 \mu(B) \geq \int_B u(x)^2 + u'(x)^2\, dx 
\end{align}
for every Borel set $B\subseteq\R$.
\end{definition}

 Associated with each pair $(u,\mu)$ in $\Dper$ is a spectral problem of the form 
\begin{align}\label{eqnISP}
 -f'' + \frac{1}{4} f = z\, \omega f + z^2 \dip f,
\end{align}
where $\omega = u-u''$ is an $\period$-periodic real distribution in $H^{-1}_{\loc}(\R)$  and the $\period$-periodic non-negative Borel measure $\dip$ on $\R$ is defined such that  
\begin{align}
  \mu(B) =  \int_B u(x)^2 + u'(x)^2\, dx  + \dip(B) 
\end{align}
for every Borel set $B\subseteq\R$.
Due to the low regularity of the coefficients, this differential equation has to be understood in a distributional sense in general (see Section~\ref{secBDE} below for details). 
The relevance of the spectral problem~\eqref{eqnISP} stems from the fact that it serves as an isospectral problem for the conservative Camassa--Holm flow; see \cite{ConservCH, LagrangeCH, ConservMP}. 

Virtually the entire literature dedicated to the periodic isospectral problem for the conservative Camassa--Holm flow deals with the case when the measure $\dip$ vanishes identically and the coefficient $\omega$ obeys additional smoothness or positivity restrictions. 
More precisely, a fairly complete spectral picture of the problem 
\begin{align}\label{eqnISPch}
 -f'' + \frac{1}{4} f = z\, \omega f
\end{align}
has been given in \cite{co97a,co97,comc99} when $\omega$ is at least a continuous function. 
Under the additional assumption that $\omega$ is sufficiently smooth and positive, the inverse spectral problem for~\eqref{eqnISPch} has been considered in \cite{baklko03, co98, ko04}. 
Somewhat related in this context, smooth finite gap solutions of the Camassa--Holm equation, respectively its two-component generalization, have been studied in \cite{comc99, CH2Real, geho08}. 
Another kind of finite gap solutions, periodic multi-peakon solutions of the form
\begin{align}
    u(x,t) = \sum_{k\in\Z} \sum_{n=1}^N p_n(t) \E^{-|x-q_n(t) - k\period|},
  \end{align}
also constitute a reasonably well-studied class. 
It has been shown in \cite{alcafehoma01, alfe01, besasz02, besasz05} that in this case, the coefficient $\omega$ in~\eqref{eqnISPch} can be recovered from the associated spectral data, which led to a representation of periodic multi-peakon solutions in terms of  theta functions. 
However, a complete solution of the corresponding inverse spectral problem in this situation has been given only recently in \cite{InvPeriodMP} by considering the generalized differential equation~\eqref{eqnISP}.
We point out again that this modification of the isospectral problem is in full accordance with the notion of global conservative solutions of the Camassa--Holm equation; see \cite{grhora13, hora08b}.

Our aim in the present article is to develop some basic direct spectral theory for the isospectral problem~\eqref{eqnISP} with periodic coefficients covering the entire natural phase space $\Dper$. 
We first start with a brief discussion of the differential equation~\eqref{eqnISP} in Section~\ref{secBDE}, including several necessary facts about solutions. 
Section~\ref{secPS} then introduces the associated spectra (periodic/antiperiodic as well as Dirichlet) and other useful quantities like the Floquet discriminant. 
In the following section, we derive trace formulas for these spectra, which play a crucial role for inverse spectral theory, as well as sharp lower bounds for (the moduli of) the periodic/antiperiodic eigenvalues. 
We then proceed to show in Section~\ref{secCE} that the spectra depend continuously on the coefficients with respect to a weak$^\ast$ topology. 
It should be mentioned here that continuous dependence and sharp lower bounds for the periodic eigenvalues have been established recently in~\cite{chmezh18} for the definite case, when the measure $\dip$ vanishes identically and the distribution $\omega$ is positive.

\subsection*{Notation} 
Let us first introduce several spaces of functions and distributions.  
We denote with $H^1_{\loc}(\R)$, $H^1(\R)$ and $H^1_{\cc}(\R)$ the usual Sobolev spaces:
\begin{align}
H^1_{\loc}(\R) & =  \lbrace f\in AC_{\loc}(\R) \,|\, f'\in L^2_{\loc}(\R) \rbrace, \\
 H^1(\R) & = \lbrace f\in H^1_{\loc}(\R) \,|\, f,\, f'\in L^2(\R) \rbrace, \\ 
 H^1_{\cc}(\R) & = \lbrace f\in H^1(\R) \,|\, \supp(f) \text{ compact} \rbrace.
\end{align}
Furthermore, the space of distributions $H^{-1}_{\loc}(\R)$ is the topological dual of $H^1_{\cc}(\R)$ and the space $H^1_{\per}(\R)$ consists of all $\period$-periodic  functions in $H^1_\loc(\R)$, that is,
 \begin{align}
H^1_{\per}(\R)  =  \lbrace f\in H^1_{\loc}(\R) \,|\, f(x+\period) = f(x)  \text{ for all } x\in\R\rbrace,
\end{align}
where $\period>0$ is a fixed given period.
 
  For integrals of a function $f$ which is locally integrable with respect to a Borel measure $\nu$ on an interval $I$, we will employ the convenient notation 
\begin{align}\label{eqnDefintmu}
 \int_x^y f\, d\nu = \begin{cases}
                                     \int_{[x,y)} f\, d\nu, & y>x, \\
                                     0,                                     & y=x, \\
                                     -\int_{[y,x)} f\, d\nu, & y< x, 
                                    \end{cases} \qquad x,\,y\in I, 
\end{align}
 rendering the integral left-continuous as a function of $y$. 
 If the function $f$ is locally absolutely continuous on $I$ and $g$ denotes a left-continuous distribution function of the measure $\nu$, then we have the integration by parts formula 
\begin{align}\label{eqnPI}
  \int_{x}^y  f\, d\nu =  \left. g f\right|_x^y - \int_{x}^y g(\xi) f'(\xi) d\xi, \quad x,\,y\in I,
\end{align}
 which will be used occasionally throughout this article. 

 Given real numbers $\eta_i$ indexed by some set $\inds$ such that there are only finitely many $\eta_i$ with $|\eta_i|\geq \varepsilon$ for every positive $\varepsilon$, we introduce the notation
\begin{align}\label{eqncsumdef}
 \csum_{i\in\inds} \eta_i = \lim_{\varepsilon\downarrow0} \mathop{\sum_{i\in\inds}}_{|\eta_i|\geq\varepsilon} \eta_i,
\end{align}
provided that the limit exists. 
Similarly, subject to existence, we shall write
\begin{align}\label{eqncproddef}
 \cprod_{i\in\inds} (1-\eta_i z) = \lim_{\varepsilon\downarrow0} \mathop{\prod_{i\in\inds}}_{|\eta_i|\geq\varepsilon} (1-\eta_i z), \quad z\in\C,
\end{align}
where the limit is meant to be taken in the topology of locally uniform convergence. 
The limit in~\eqref{eqncproddef} exists if and only if the limit in~\eqref{eqncsumdef} exists and the sum
\begin{align}
 \sum_{i\in\inds} \eta_i^2
\end{align}
is finite. 
In this case, upon denoting  the entire function in~\eqref{eqncproddef}  with $\Lambda$, we have 
\begin{align}\label{eqncprodatzero}
 \csum_{i\in\inds} \eta_i & = - \dot{\Lambda}(0), & \sum_{i\in\inds} \eta_i^2 & = \dot{\Lambda}(0)^2 - \ddot{\Lambda}(0).
\end{align}

\section{The basic differential equation}\label{secBDE}

Throughout this article, we fix a period $\ell>0$ and let $(u,\mu)$ be a pair in $\Dper$. 
 Associated with such a pair  is an $\period$-periodic distribution $\omega$ in $H^{-1}_{\loc}(\R)$ defined by 
\begin{align}\label{eqnDefomega}
 \omega(h) = \int_\R u(x)h(x)dx + \int_\R u'(x)h'(x)dx, \quad h\in H^1_c(\R),
\end{align}
so that $\omega = u - u''$ in a distributional sense, as well as an $\period$-periodic non-negative Borel measure $\dip$ on $\R$ that satisfies
\begin{align}\label{eqnDefdip}
  \mu(B) =  \int_B u(x)^2 + u'(x)^2\, dx  + \dip(B) 
\end{align}
for every Borel set $B\subseteq\R$. 
We consider the differential equation 
 \begin{align}\label{eqnDEho}
  -f'' + \frac{1}{4} f = z\, \omega f + z^2 \dip f, 
 \end{align}
 where $z$ is a complex spectral parameter. 
 Of course, this differential equation has to be understood in a distributional way in general; compare \cite{ConservCH, IndefiniteString, gewe14, sash03}.   
 All of the facts stated without proofs in this section may be found in \cite[Appendix~A]{ConservCH}, \cite[Appendices~A and~B]{CHTrace}, \cite{CHPencil}.  
  
  \begin{definition}\label{defSolution}
  A solution of~\eqref{eqnDEho} is a function $f\in H^1_{\loc}(\R)$ such that 
 \begin{align}\label{eqnODEweak}
   \int_{\R} f'(x) h'(x) dx + \frac{1}{4} \int_\R f(x)h(x)dx = z\, \omega(fh) + z^2 \int_\R f h \,d\dip 
 \end{align} 
 for every function $h\in H^1_\cc(\R)$.
 \end{definition}
 
Let us emphasize here that, in general, the derivative of such a solution will only be a locally square integrable function.

\begin{lemma}\label{lemEquQD}
 If the function $f$ is a solution of the differential equation~\eqref{eqnDEho}, then there is a unique left-continuous function $f^\qd$ such that 
 \begin{align}\label{eqnfqd} 
     f^\qd = f' - z u' f
\end{align} 
almost everywhere. 
The function $f^\qd$ satisfies 
  \begin{align}\label{eqnDEfqd}
    \left. f^\qd \right|_x^y & =  \frac{1}{4} \int_x^y f(\xi)d\xi  -z\int_x^y u(\xi)f(\xi)+u'(\xi)f'(\xi)\,d\xi  - z^2 \int_x^y f \,d\dip 
  \end{align}
  for all $x,\,y\in\R$.
 \end{lemma}
 
 We will refer to the function $f^\qd$ as the {\em quasi-derivative} of a solution $f$.
 A simple integration by parts turns~\eqref{eqnDEfqd} in Lemma~\ref{lemEquQD} into the useful identity 
 \begin{align}\begin{split}\label{eqnDEint}
   & \int_x^y f'(\xi)h'(\xi)d\xi + \frac{1}{4} \int_x^y f(\xi)h(\xi)d\xi \\
   & \qquad = \left.f^\qd h\right|_{x}^{y} + z\int_x^y u(\xi)f(\xi)h(\xi) + u'(\xi)(fh)'(\xi)\,d\xi + z^2 \int_x^y fh\, d\dip
 \end{split}\end{align}
 for all $x$, $y\in\R$ as long as the function $h$ belongs to $H^1_\loc(\R)$.
 The quasi-derivative also allows us to state existence and uniqueness results for initial value problems.
  
\begin{corollary}\label{corSolIVP}
 For every $a\in\R$ and $d_1$, $d_2\in\C$, there is a unique solution $f$ of the differential equation~\eqref{eqnDEho} with the initial conditions
 \begin{align}\label{eqnSolIV}
 f(a) & = d_1, &  f^\qd(a) & = d_2.
 \end{align} 
 If $d_1$, $d_2$ and $z$ are real, then the solution $f$ is real-valued as well. 
\end{corollary}
    
 It is also possible to introduce a Wronskian $W(f,g)$ of two solutions $f$ and $g$. 

\begin{corollary}\label{corWronskian}
 For any two solutions $f$, $g$ of the differential equation~\eqref{eqnDEho}, there is a unique complex number $W(f,g)$ such that 
\begin{align}
W(f,g) = f(x)g'(x) -f'(x)g(x)
\end{align} 
 for almost all $x\in\R$. 
 The number $W(f,g)$ is zero if and only if the functions $f$ and $g$ are linearly dependent.
\end{corollary}

  We proceed with a result about continuous dependence of solutions to initial value problems of the form in Corollary~\ref{corSolIVP} on the coefficients of the differential equation. 
  In order to state it, let $(u_k,\mu_k)$ be a sequence of pairs in $\Dper$ and define corresponding distributions $\omega_k$ in $H^{-1}_{\loc}(\R)$ and non-negative Borel measures $\dip_k$ in the same way as for $(u,\mu)$. 
  
 \begin{lemma}\label{lemAppIVPsolcon}
  Fix $a\in\R$, $d_1$, $d_2\in\C$ and let $f$ be the solution of the differential equation~\eqref{eqnDEho} with the initial conditions~\eqref{eqnSolIV}.      
  Suppose that the measure $\dip$ has no mass in $a$, that the functions $u_k$ converge to $u$ weakly in $H^1_\loc(\R)$\footnote{A sequence of functions $g_k$ converges to some $g$ weakly in $H^1_\loc(\R)$ if and only if the restricted functions $g_k|_I$ converge to $g|_I$ weakly in $H^1(I)$ for every compact interval $I$.} and that
  \begin{align}
    \int_\R  u_k'(x)^2 h(x) dx +  \int_\R h\, d\dip_k \rightarrow  \int_\R u'(x)^2 h(x) dx +  \int_\R h\, d\dip
  \end{align}
  for all functions $h\in C_\cc(\R)$.
  Then the solutions $f_k$ of the differential equation 
     \begin{align}\label{eqnDEhok}
  -f_k'' + \frac{1}{4} f_k = z\, \omega_k f_k + z^2 \dip_k f_k 
 \end{align}
  with the initial conditions 
   \begin{align}
   f_k(a) & = d_1, &  f_k^\qd(a) & = d_2,
 \end{align} 
  converge to $f$ weakly in $H^1_\loc(\R)$ and the quasi-derivatives $f_k^\qd$ converge to $f^\qd$ in $L^2_\loc(\R)$ as well as in every point where the measure $\dip$ has no mass.
 \end{lemma}
    
  Let us fix an arbitrary base point $a\in\R$ and consider the fundamental system of solutions $c(z,\redot)$, $s(z,\redot)$ of the differential equation~\eqref{eqnDEho} with the initial conditions
 \begin{align}\label{eqnFSentIC}
   c(z,a) & = s^\qd(z,a) = 1, & c^\qd(z,a) &  = s(z,a) = 0, 
 \end{align} 
 for every $z\in\C$. Note that when $z$ is zero, we have explicitly 
\begin{align}\label{eqnCSatzero}
c(0,x) & = \cosh\Bigl(\frac{x-a}{2}\Bigr), & s(0,x) & = 2 \sinh\Bigl(\frac{x-a}{2}\Bigr), \qquad x\in\R.
\end{align}
 We are now going to collect some further useful properties of these solutions. 
 
 \begin{proposition}\label{propcsent}
  For each $x\in\R$, the functions 
  \begin{align}\label{eqnFSent}
   z & \mapsto c(z,x), & z & \mapsto c^\qd(z,x), & z & \mapsto s(z,x), & z & \mapsto s^\qd(z,x),
  \end{align}
    are real entire of Cartwright class with only nonzero real and simple roots (as long as $x$ differs from $a$)  and for every bounded set $K\subset\R$ there are constants $A$, $B\in\R$ such that 
   \begin{align}
    |c(z,x)|,\, |c^\qd(z,x)|,\, |s(z,x)|,\, |s^\qd(z,x)| \leq A\E^{B|z|},  \quad z\in\C,~x\in K.
   \end{align}
    Moreover, we have 
  \begin{align}
    \label{eqndc0} \dot{c}(0,x) & =  \cosh\Bigl(\frac{x-a}{2}\Bigr) (u(x)-u(a)) -\sinh\Bigl(\frac{x-a}{2}\Bigr) \int_{a}^x u(\xi)d\xi, \\
    \label{eqndcp0} \dot{c}^\qd(0,x) & =  -\frac{1}{2} \sinh\Bigl(\frac{x-a}{2}\Bigr) (u(x)+u(a)) - \frac{1}{2} \cosh\Bigl(\frac{x-a}{2}\Bigr) \int_{a}^x u(\xi)d\xi, 
   \end{align}
  and
  \begin{align}   \label{eqnds0} \dot{s}(0,x) & = 2\sinh\Bigl(\frac{x-a}{2}\Bigr) (u(x)+u(a)) - 2 \cosh\Bigl(\frac{x-a}{2}\Bigr) \int_{a}^x u(\xi)d\xi, \\
    \label{eqndsp0} \dot{s}^\qd(0,x) & = - \cosh\Bigl(\frac{x-a}{2}\Bigr) (u(x)-u(a)) - \sinh\Bigl(\frac{x-a}{2}\Bigr) \int_{a}^x u(\xi)d\xi,
  \end{align}
  as well as  
  \begin{align}
    \begin{split}\label{eqnddc0}  \ddot{c}(0,x) & = \cosh\Bigl(\frac{x-a}{2}\Bigr) (u(x)-u(a))^2 + \cosh\Bigl(\frac{x-a}{2}\Bigr) \biggl(\int_{a}^x u(\xi)d\xi\biggr)^2 \\
                                                    & \quad - 2\sinh\Bigl(\frac{x-a}{2}\Bigr) (u(x)-u(a))\int_{a}^x u(\xi)d\xi  - 2\sinh\Bigl(\frac{x-a}{2}\Bigr) \int_{a}^x d\mu \\
                                                    & \quad + 2\int_{a}^x u(\xi)^2 \sinh\Bigl(\xi-\frac{x+a}{2}\Bigr) d\xi + 2\int_{a}^x \sinh\Bigl(\xi - \frac{x+a}{2}\Bigr) d\mu(\xi),                                                     
    \end{split}  \\
    \begin{split}
    \label{eqnddcp0} \ddot{c}^\qd(0,x) & = \frac{1}{2} \sinh\Bigl(\frac{x-a}{2}\Bigr) (u(x)+u(a))^2 + \frac{1}{2} \sinh\Bigl(\frac{x-a}{2}\Bigr) \biggl(\int_{a}^x u(\xi)d\xi\biggr)^2  \\
                                                    & \quad + \cosh\Bigl(\frac{x-a}{2}\Bigr) (u(x)+u(a))\int_{a}^x u(\xi)d\xi  - \cosh\Bigl(\frac{x-a}{2}\Bigr) \int_{a}^x d\mu \\
                                                    & \quad -  \int_{a}^x u(\xi)^2 \cosh\Bigl(\xi-\frac{x+a}{2}\Bigr) d\xi - \int_{a}^x \cosh\Bigl(\xi - \frac{x+a}{2}\Bigr) d\mu(\xi),                                                    
    \end{split} 
      \end{align}
  and
  \begin{align}
    \begin{split}
    \label{eqndds0} \ddot{s}(0,x) & = 2\sinh\Bigl(\frac{x-a}{2}\Bigr) (u(x)+u(a))^2 + 2\sinh\Bigl(\frac{x-a}{2}\Bigr) \biggl(\int_{a}^x u(\xi)d\xi\biggr)^2  \\
                                                    & \quad - 4\cosh\Bigl(\frac{x-a}{2}\Bigr) (u(x)+u(a))\int_{a}^x u(\xi)d\xi  - 4 \cosh\Bigl(\frac{x-a}{2}\Bigr) \int_{a}^x d\mu \\
                                                    & \quad + 4 \int_{a}^x u(\xi)^2 \cosh\Bigl(\xi-\frac{x+a}{2}\Bigr) d\xi + 4 \int_{a}^x \cosh\Bigl(\xi - \frac{x+a}{2}\Bigr) d\mu(\xi),                                                         
    \end{split} \\
    \begin{split}
    \label{eqnddsp0} \ddot{s}^\qd(0,x) & = \cosh\Bigl(\frac{x-a}{2}\Bigr) (u(x)-u(a))^2 + \cosh\Bigl(\frac{x-a}{2}\Bigr) \biggl(\int_{a}^x u(\xi)d\xi\biggr)^2 \\
                                                    & \quad + 2\sinh\Bigl(\frac{x-a}{2}\Bigr) (u(x)-u(a))\int_{a}^x u(\xi)d\xi  - 2\sinh\Bigl(\frac{x-a}{2}\Bigr) \int_{a}^x d\mu \\
                                                    & \quad - 2\int_{a}^x u(\xi)^2 \sinh\Bigl(\xi-\frac{x+a}{2}\Bigr) d\xi - 2\int_{a}^x \sinh\Bigl(\xi - \frac{x+a}{2}\Bigr) d\mu(\xi).                                                 
   \end{split} \end{align}
 \end{proposition}
  
   We note that differentiation with respect to the spectral parameter is denoted with a dot and always meant to be done after taking quasi-derivatives.
 
  \begin{remark}\label{remET}
  It is possible to show that the exponential types of the entire functions in~\eqref{eqnFSent} are all the same and given by the modulus of the integral
  \begin{align}\label{eqnET}
    \int_{a}^x \rho(\xi) d\xi,
  \end{align}
  where $\rho$ is the square root of the Radon--Nikod\'ym derivative of the absolutely continuous part of the measure $\dip$ (with respect to the Lebesgue measure). 
 \end{remark}

 \section{Spectra}\label{secPS}

For an arbitrary but fixed base point $a\in\R$, we first introduce the {\em monodromy matrix} $M$ via 
 \begin{equation}\label{eq:cM}
 M(z)=\begin{pmatrix}
 c(z,a+\period) & s(z,a+\period)\\
 c^\qd(z,a+\period) & s^\qd(z,a+\period) 
 \end{pmatrix},\quad z\in\C, 
 \end{equation}
 as well as the {\em Floquet discriminant} $\Delta$  by 
 \begin{equation}\label{eq:Delta}
 \Delta(z)=\frac{\tr\,M(z)}{2} =\frac{c(z,a+\period)+s^\qd(z,a+\period)}{2},\quad z\in\C.
 \end{equation} 
 It follows readily from Lemma~\ref{lemEquQD} and Corollary~\ref{corWronskian} that the determinant of $M(z)$ equals one for all $z\in\C$.  
 The qualitative behavior of the real entire function $\Delta$ on the real line is captured by the following result. 
 
 \begin{lemma}\label{lemDelta}
 All zeros of $\Delta$ are real, non-zero and simple.
 Moreover, if $\dot{\Delta}(z)=0$ for some $z\in\R$, then $|\Delta(z)|\ge 1$ and $\Delta(z)\ddot{\Delta}(z)<0$ unless $\Delta$ is constant. 
 \end{lemma} 
 
 \begin{proof}
   Let us suppose that $z\in\C$ is a zero of the Floquet discriminant $\Delta$.
   Since the determinant of $M(z)$ equals one, this implies that $\I$ is an eigenvalue of the matrix $M(z)$. 
   Thus there is a nontrivial solution $f$ of the differential equation~\eqref{eqnDEho} such that $f(a+\period) = \I f(a)$ and $f^\qd(a+\period) = \I f^\qd(a)$. 
   Setting $h=z^\ast f^\ast$, $x=a$ and $y=a+\period$ in~\eqref{eqnDEint} before taking the imaginary part then gives 
   \begin{align}\label{eqnfSA}
       \biggl( \int_a^{a+\period} |f'(\xi)|^2 d\xi +  \frac{1}{4}  \int_a^{a+\period} |f(\xi)|^2d\xi + \int_a^{a+\period} |z f|^2 d\dip\biggr) \im\,z =0.
   \end{align}
   As the expression in the brackets is positive, we may conclude that $z$ is real. 
   
    In order to compute the derivative of the Floquet discriminant $\Delta$, we first introduce the solutions $c_\period(z,\redot)$, $s_\period(z,\redot)$ of the differential equation~\eqref{eqnDEho} given by 
    \begin{align}\begin{split}\label{eqncsper}
      c_\period(z,x) & = s^\qd(z,a+\period)c(z,x) - c^\qd(z,a+\period)s(z,x), \\ s_\period(z,x) & = -s(z,a+\period)c(z,x) + c(z,a+\period)s(z,x), 
    \end{split}\end{align}
    for every $x\in\R$ and $z\in\C$, so that they satisfy the initial conditions  
   \begin{align*}
     c_\period(z,a+\period) & = s_\period^\qd(z,a+\period) = 1, &  c_\period^\qd(z,a+\period) & = s_\period(z,a+\period)=0,
   \end{align*}
   at the point $a+\period$.
   Upon choosing $f=c(z,\redot)$ in~\eqref{eqnDEfqd} and replacing $f'$ using the quasi-derivative, we differentiate with respect to $z$ to obtain
\begin{align*}
  \left. \dot{c}^\qd(z,\redot)\right|_x^y  & = \frac{1}{4} \int_x^y \dot{c}(z,\xi)d\xi -   z\int_x^y u(\xi)\dot{c}(z,\xi) + u'(\xi)\dot{c}^\qd(z,\xi)\, d\xi  \\ 
   & \qquad - \int_x^y u(\xi) c(z,\xi) + u'(\xi)c^\qd(z,\xi)\, d\xi - 2z \int_x^y u'(\xi)^2 c(z,\xi) d\xi \\
   & \qquad  - z^2 \int_x^y u'(\xi)^2 \dot{c}(z,\xi)d\xi - 2z\int_x^y c(z,\xi)d\dip(\xi) - z^2 \int_x^y \dot{c}(z,\xi)d\dip(\xi)
 \end{align*} 
for every $x$, $y\in\R$ and $z\in\C$. 
Employing the integration by parts formula~\eqref{eqnPI}, this gives the identity
   \begin{align*}
    \dot{c}(z,a+\period) & = \left. \dot{c}(z,\redot) s_\period^\qd(z,\redot) - \dot{c}^\qd(z,\redot) s_\period(z,\redot) \right|_a^{a+\period} =\phi_z(c(z,\redot)s_\period(z,\redot)), \quad z\in\C, 
   \end{align*} 
 after some calculations, where $\phi_z(f)$ is defined by  
\begin{align*}
\phi_z(f) = \int_a^{a+\period} u(x){f}(x) + u'(x){f}'(x)\, dx + 2z\int_a^{a+\period} f\, d\dip, \quad f\in H^1_\loc(\R).
\end{align*}
   Furthermore, in much the same manner, we also obtain
   \begin{align*}
    -\dot{s}^\qd(z,a+\period) & =  \phi_z(c_\period(z,\redot)s(z,\redot)), \quad z\in\C.
   \end{align*}   
   After plugging in~\eqref{eqncsper}, these two equations add up to
   \begin{align}\begin{split}\label{eqnDeltadot}
     \dot{\Delta}(z) & = \frac{c^\qd(z,a+\period)}{2} \phi_z(s(z,\redot)^2)  - \frac{s(z,a+\period)}{2} \phi_z(c(z,\redot)^2) \\
                              & \qquad\qquad + \frac{c(z,a+\period) - s^\qd(z,a+\period)}{2} \phi_z(c(z,\redot)s(z,\redot)), \quad z\in\C.
   \end{split}\end{align}
   Moreover, as long as $s(z,a+\period)$ is non-zero, we get 
   \begin{align}\label{eqnDeltadotpsi}
     - \frac{2\dot{\Delta}(z)}{s(z,a+\period)}  & =  \phi_z( \psi_-(z,\redot) \psi_+(z,\redot)), 
   \end{align} 
   where $\psi_-(z,\redot)$ and $\psi_+(z,\redot)$ are the (nontrivial) solutions of~\eqref{eqnDEho} given by
   \begin{align*}
    \psi_\pm(z,x) = c(z,x) + \frac{\Delta(z)  - c(z,a+\period)  \pm \sqrt{\Delta(z)^2-1}}{s(z,a+\period)} s(z,x), \quad x\in\R.
   \end{align*}    
   
   Now suppose that $\dot{\Delta}(z)=0$ for some $z\in\R$. 
   If $s(z,a+\period)$ is zero, then 
   \begin{align*}
     c(z,a+\period)s^\qd(z,a+\period) = \det M(z) = 1
    \end{align*}
     and therefore 
   \begin{align*}
    2 |\Delta(z)| =  \bigl| c(z,a+\period) + s^\qd(z,a+\period)\bigr| =  \bigl| c(z,a+\period) + c(z,a+\period)^{-1} \bigr| \geq 2.
   \end{align*}
   Otherwise $s(z,a+\period)$ is non-zero and thus the right-hand side of~\eqref{eqnDeltadotpsi} equals zero.
   Now if $|\Delta(z)|$ was less than one, then we would have $\psi_-(z,\redot)=\psi_+(z,\redot)^\ast$ and using~\eqref{eqnDEint}, the right-hand side of~\eqref{eqnDeltadotpsi} would turn into 
   \begin{align*}
                \frac{1}{z} \biggl(\int_a^{a+\period} \left|\psi_+'(z,x)\right|^2 dx + \frac{1}{4} \int_a^{a+\period} \left|\psi_+(z,x)\right|^2 dx +  \int_a^{a+\period} \left|z\psi_+(z,x)\right|^2 d\dip(x)\biggr)\not=0,
   \end{align*}   
   which constitutes a contradiction. 
   This shows that $|\Delta(z)|\geq 1$ in either case. 
   In particular, it guarantees that all zeros of $\Delta$ are simple and we are left to note that the last claim is true for any non-constant real entire function of Cartwright class with only real and simple zeros. 
 \end{proof}

Let us now consider the spectral problem associated with our differential equation~\eqref{eqnDEho} on the interval $[a,a+\period)$ with periodic/antiperiodic boundary conditions.
The corresponding periodic/antiperiodic spectrum $\sigma_{\pm}$ consist of all those $z\in\C$ for which there is a nontrivial solution $f$ of the differential equation~\eqref{eqnDEho} with 
\begin{align}\label{eqnPBC}
 \begin{pmatrix} f(a) \\ f^\qd(a) \end{pmatrix} = \pm \begin{pmatrix} f(a+\period) \\ f^\qd(a+\period) \end{pmatrix}. 
\end{align}
Under the multiplicity of a periodic/antiperiodic eigenvalue we understand the number of linearly independent solutions of~\eqref{eqnDEho} that satisfy~\eqref{eqnPBC}.

  \begin{proposition}\label{propperspec}
   The periodic/antiperiodic spectrum $\sigma_\pm$ is a discrete set of non\-zero reals and coincides with the set of zeros of the entire  function $\Delta\mp1$. 
   Each periodic/anti\-periodic eigenvalue's multiplicity is equal to its multiplicity as a zero of $\Delta\mp1$. 
  \end{proposition}
  
  \begin{proof}
   Let $z\in\sigma_\pm$ be a periodic/antiperiodic eigenvalue with a corresponding eigenfunction $f$.
   If $z$  was zero, then setting $h= f^\ast$, $x=a$ and $y=a+\period$ in~\eqref{eqnDEint} would yield the contradiction
   \begin{align*}
     \int_a^{a+\period} |f'(\xi)|^2 d\xi + \frac{1}{4} \int_a^{a+\period} |f(\xi)|^2 d\xi = \left. f^\qd f^\ast\right|_a^{a+\period} = 0. 
   \end{align*} 
   Moreover, upon setting $h=z^\ast f^\ast$, $x=a$ and $y=a+\period$ in~\eqref{eqnDEint} and taking the imaginary part gives~\eqref{eqnfSA}, which shows that $z$ has to be real. 
   Furthermore, it is readily seen that some $z\in\C$ is a periodic/antiperiodic eigenvalue if and only if $\pm 1$ is an eigenvalue of the matrix $M(z)$, which is equivalent to $\Delta(z) = \pm1$. 
   In particular, this also shows that the spectrum $\sigma_\pm$ is a discrete set. 
   
   Since the zeros of $\Delta\mp1$ have multiplicity at most two by Lemma~\ref{lemDelta}, it remains to show that a periodic/antiperiodic eigenvalue $z\in\sigma_\pm$ is double if and only if $\dot{\Delta}(z)$ vanishes. 
   If the eigenvalue $z$ has multiplicity two, then $\pm M(z)$ is the identity matrix and~\eqref{eqnDeltadot} shows that $\dot{\Delta}(z)$ is zero. 
   Conversely, if we suppose that $\dot{\Delta}(z)$ vanishes, then $s(z,a+\period)$ has to be zero, since otherwise~\eqref{eqnDeltadotpsi} would give a contradiction. 
   It then follows readily that also $c(z,a+\period)=s^\qd(z,a+\period)=\pm1$ and consequently~\eqref{eqnDeltadot} implies that $c^\qd(z,a+\period)=0$ and thus $\pm M(z)$ is the identity matrix. 
  \end{proof}
     
 \begin{remark}\label{remBaseInd}
  It is not difficult to see that the periodic/antiperiodic spectrum (including multiplicities) is independent of the chosen base point $a$ and thus so is the Floquet discriminant.
\end{remark} 
 
 Let us now consider the zeros of the entire function $\Delta^2 - 1$, each of which is non-zero, real and has multiplicity at most two in view of Lemma~\ref{lemDelta}. 
 If there are only finitely many positive zeros, then their number (counted with multiplicities) is even since $\Delta^2-1$ is positive at zero as well as for large enough positive values. 
 Thus we may label them in non-decreasing order by 
 \begin{align}
  \lambda_1,\lambda_2,\ldots,\lambda_{2I_+ -1},\lambda_{2I_+}
 \end{align}
 for some non-negative integer $I_+$. 
 When there are infinitely many positive zeros, we set $I_+=\infty$ and label them in non-decreasing order by 
  \begin{align}
  \lambda_1,\lambda_2,\lambda_3,\lambda_4,\ldots
 \end{align}
 taking also into account multiplicities again. 
 In a similar way, if there are only finitely many negative zeros, then their number (counted with multiplicities) is even since $\Delta^2-1$ is positive at zero as well as for large enough negative values. 
 Thus we may label them in non-decreasing order by 
 \begin{align}
  \lambda_{-2I_-},\lambda_{-2I_-+1},\ldots,\lambda_{-2},\lambda_{-1}
 \end{align}
 for some non-negative integer $I_-$. 
 When there are infinitely many negative zeros, we set $I_-=\infty$ and label them in non-decreasing order by 
  \begin{align}
  \ldots,\lambda_{-4},\lambda_{-3},\lambda_{-2},\lambda_{-1}
 \end{align}
 taking also into account multiplicities again. 
 It follows from Lemma~\ref{lemDelta} that these sequences indeed satisfy the inequalities  
\begin{align}
   \cdots \leq \lambda_{-4} < \lambda_{-3} \le \lambda_{-2}  < \lambda_{-1} < 0  < \lambda_1 < \lambda_2 \leq  \lambda_3 < \lambda_4 \leq \cdots,
\end{align}
where the smallest (in modulus) positive and negative zero is a simple periodic eigenvalue, followed by alternating pairs (except for the last zero when there are only finitely many) of antiperiodic and periodic eigenvalues.  
 Upon introducing the index set $\inds = \lbrace i\in\Z\backslash\{0\}\,|\, -I_-\leq i\leq I_+\rbrace$, we define the intervals
 \begin{align}
 \Gamma_i=\begin{cases} [-\infty,\lambda_{2i}], & i= -I_-, \\
                                           [\lambda_{2i-1},\lambda_{2i}], & -I_- < i  <0, \\
                                            [\lambda_{2i},\lambda_{2i+1}], & 0<i < I_+, \\
                                             [\lambda_{2i},\infty], & i=I_+, \\  \end{cases} 
 \end{align}
 for each $i\in\inds$,  called {\em the gaps}. 
 A gap $\Gamma_i$ is called {\em closed} if it reduces to a single point and {\em open} otherwise. 
 If they exist, the last positive gap $\Gamma_{I_+}$ is called the {\em outermost positive gap} and the last negative gap $\Gamma_{-I_-}$ is called the {\em  outermost negative gap}. 
 The typical behavior of the Floquet discriminant $\Delta$ and the location of the gaps relative to it is depicted below:
  
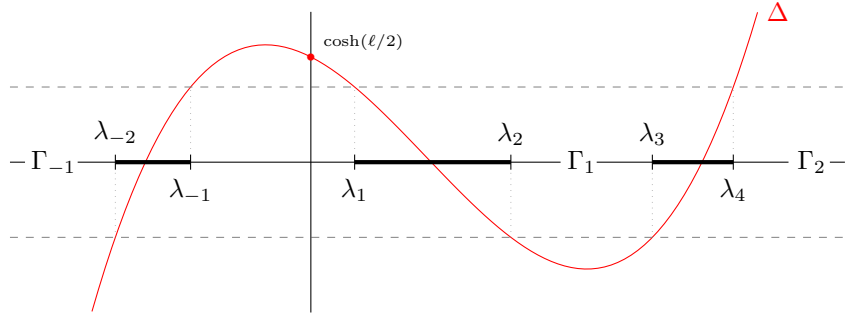
\begin{center}
 \begin{tikzpicture}[domain=-4:8, samples=101]

\draw[color=red, domain=-2.91:5.94] plot (\x,{ 1.4*(1-\x/1.6)*(1-\x/5.2)*(1+\x/2.2) }) node[right] {$\Delta$};

\draw[-] (-4,0) -- (7.1,0) node[above] {};
\draw[dashed, help lines] (-4,1) -- (7.1,1) node[right] {};
\draw[dashed, help lines] (-4,-1) -- (7.1,-1) node[right] {};
\draw[-] (0,-2) -- (0,2) node[right] {};

\draw[-] (-2.599,-0.1) -- (-2.599,0.1) node[above] {$\lambda_{-2}$};
\draw[-] (-1.599,0.1) -- (-1.599,-0.1) node[below] {$\lambda_{-1}$};
\draw[-] (0.582,0.1) -- (0.582,-0.1) node[below] {$\lambda_1$};
\draw[-] (2.659,-0.1) -- (2.659,0.1) node[above] {$\lambda_2$};
\draw[-] (4.54,-0.1) -- (4.54,0.1) node[above] {$\lambda_3$};
\draw[-] (5.616,0.1) -- (5.616,-0.1) node[below] {$\lambda_4$};

\draw[dotted, help lines] (-2.599,-1) -- (-2.599,-0.1) node[right] {};
\draw[dotted, help lines] (-1.599,0.1) -- (-1.599,1) node[right] {};
\draw[dotted, help lines] (0.582,0.1) -- (0.582,1) node[right] {};
\draw[dotted, help lines] (2.659,-1) -- (2.659,-0.1) node[right] {};
\draw[dotted, help lines] (4.54,-1) -- (4.54,-0.1) node[right] {};
\draw[dotted, help lines] (5.616,0.1) -- (5.616,1) node[right] {};

\draw[-, ultra thick] (-2.599,0) -- (-1.599,0) node[right] {};
\draw[-, ultra thick] (0.582,0) -- (2.659,0) node[right] {};
\draw[-, ultra thick] (4.54,0) -- (5.616,0) node[right] {};

\draw[-, color=white, ultra thick] (-3.8,0) -- (-3.1,0) node[above] {};
\draw[-, color=white, ultra thick] (3.32,0) -- (3.8,0) node[above] {};
\draw[-, color=white, ultra thick] (6.35,0) -- (6.9,0) node[above] {};
\node at (-3.42,-0.009) {$\Gamma_{-1}$}; 
\node at (3.6,0) {$\Gamma_{1}$};
\node at (6.66,0) {$\Gamma_{2}$};

\node[color=red] at (0,1.4) {\tiny $\bullet$};
\node at (0.7,1.6) {\tiny $\cosh(\period/2)$};

\end{tikzpicture}
\end{center}

We will next turn to the spectral problem associated with our differential equation~\eqref{eqnDEho} on the interval $[a,a+\period)$ with Dirichlet boundary conditions at the endpoints.
The corresponding spectrum $\sigma$ consists of all those $z\in\C$ for which there is a nontrivial solution $f$ of the differential equation~\eqref{eqnDEho} with $f(a)=f(a+\period)=0$.
From unique solvability of initial value problems for our differential equation in Corollary~\ref{corSolIVP}, we see that such a solution is always unique up to scalar multiples. 
 
\begin{proposition}\label{prop:s}
 The Dirichlet spectrum $\sigma$ is a discrete set of nonzero reals and coincides with the set of zeros of the entire function $s(\ledot,a+\period)$, all of which are simple.  
\end{proposition}
 
\begin{proof}
 In view of Proposition~\ref{propcsent}, it suffices to notice that some $z\in\C$ is a Dirichlet eigenvalue if and only if $s(z,a+\period)$ vanishes. 
\end{proof}

The location of the Dirichlet spectrum relative to the periodic and antiperiodic spectrum can be described by the following result. 

\begin{lemma}
 Each Dirichlet eigenvalue belongs to one of the gaps.  
 Conversely, except for the outermost gaps, each gap contains exactly one Dirichlet eigenvalue and each outermost gap contains at most one Dirichlet eigenvalue. 
\end{lemma}

\begin{proof}
 For every $z\in\C\backslash\R$, we consider the solution $\theta_\pm(z,\redot)$ of~\eqref{eqnDEho} given by 
 \begin{align*}
  \theta_\pm(z,x) = \frac{s(z,a+\period)c(z,x) - (c(z,a+\period)\mp1)s(z,x)}{zs(z,a+\period)}, \quad x\in\R,
 \end{align*}
 and then define the analytic function $M_\pm$ on $\C\backslash\R$ by  
 \begin{align*}
   M_\pm(z) = \theta_\pm^\qd(z,a)z^\ast\theta_\pm(z,a)^\ast - \theta_\pm^\qd(z,a+\period)z^\ast\theta_\pm(z,a+\period)^\ast, \quad z\in\C\backslash\R.
 \end{align*}
 Setting $f=\theta_\pm(z,\redot)$, $h=z^\ast\theta_\pm(z,\redot)^\ast$, $x=a$ and $y=a+\period$ in~\eqref{eqnDEint}, we obtain 
 \begin{align*}
  \frac{M_\pm(z)-M_\pm(z)^\ast}{z-z^\ast}  & = \int_a^{a+\period} |\theta_\pm'(z,\xi)|^2 d\xi +  \frac{1}{4} \int_a^{a+\period} |\theta_\pm(z,\xi)|^2 d\xi \\
                        & \qquad\qquad\qquad\qquad\quad  + \int_a^{a+\period} |z\theta_\pm(z,\xi)|^2 d\dip(\xi), \quad z\in\C\backslash\R.
 \end{align*}
 Thus the function $M_\pm$ is a Herglotz--Nevanlinna function and upon noting that 
 \begin{align*}
   M_\pm(z) = -\frac{2\Delta(z)\mp2}{zs(z,a+\period)}, \quad z\in\C\backslash\R,
 \end{align*}
 the claims follow from the entailing interlacing property of zeros and poles.
\end{proof}

Let us finally define a strictly increasing sequence $\kappa_i\in\Gamma_i$, indexed by $i\in\inds$, in the following way: 
For every $i\in\inds$ such that there is a Dirichlet eigenvalue in the gap $\Gamma_i$, we define $\kappa_i$ to be this (uniquely determined) eigenvalue.
If there is no Dirichlet eigenvalue in the gap $\Gamma_i$, then we define $\kappa_i$ to be $-\infty$ if $\Gamma_i$ is the outermost negative gap and $\kappa_i$ to be $\infty$ if $\Gamma_i$ is the outermost positive gap.
Clearly, the Dirichlet spectrum $\sigma$ consists precisely of all those $\kappa_i$ that are finite.

 \section{Trace formulas}\label{secTF}

 We are now going to collect some trace formulas, which provide relations between the pair $(u,\mu)$ and the periodic/antiperiodic spectrum as well as the Dirichlet spectrum. 
 To this end, let us first enumerate the non-decreasing sequence of periodic/antiperiodic eigenvalues (including multiplicities) as  $\lambda_{i}^\pm$ with index $i\in\inds$ in such a way that $\lambda_i^\pm$ has the same sign as $i$.  
  By means of the Cartwright--Levinson theorem \cite[Lecture~17]{le96}, we thus have the product representation 
   \begin{align}\label{eqnDeltaprod}
   \Delta(z) \mp 1 = (\cosh(\period/2) \mp 1) \cprod_{i\in\inds} \biggl(1-\frac{z}{\lambda_{i}^\pm}\biggr),    \quad z\in\C,
  \end{align}   
 for the entire function $\Delta\mp 1$.
 This allows us to derive trace formulas, which will reappear as conserved quantities for the periodic conservative Camassa--Holm flow.
 
  \begin{proposition}\label{proptraceformulas}
   The first two trace formulas are
     \begin{align}
    \label{eqnTF1}\csum_{i\in\inds}\frac{1}{\lambda_{i}^\pm} & = \frac{\sinh(\period/2)}{\cosh(\period/2) \mp 1} \int_a^{a+\period} u(x) dx, \\ 
    \label{eqnTF2} \sum_{i\in\inds} \frac{1}{(\lambda_{i}^{\pm})^2} & = \frac{\pm 1}{\cosh(\period/2) \mp 1} \biggl(\int_a^{a+\period} u(x)dx\biggr)^2 
    + \frac{2\sinh(\period/2)}{\cosh(\period/2) \mp 1} \int_a^{a+\period} d\mu. 
 \end{align}
 \end{proposition} 
 
 \begin{proof}
  By using the identities in Proposition~\ref{propcsent}, we readily obtain 
  \begin{align}
   \label{eqndDel0}\dot{\Delta}(0) & =  -\sinh(\period/2) \int_a^{a+\period} u(x) dx, \\
   \label{eqnddDel0}\ddot{\Delta}(0) & = \cosh(\period/2) \biggl(\int_a^{a+\period} u(x)dx\biggr)^2 - 2 \sinh(\period/2) \int_a^{a+\period} d\mu.
  \end{align}
  Now the claim follows from the formulas~\eqref{eqncprodatzero} applied to the product~\eqref{eqnDeltaprod}. 
 \end{proof}
 
Alternatively, we can add up the individual trace formulas for the periodic eigenvalues and for the antiperiodic eigenvalues to get the identities
    \begin{align}
     \label{eqnTF1c}\csum_{i\in\inds}\frac{1}{\lambda_{i}^+} + \csum_{i\in\inds}\frac{1}{\lambda_{i}^-} & = 2 \coth(\period/2)\int_a^{a+\period} u(x) dx, \\
     \label{eqnTF2c} \sum_{i\in\inds} \frac{1}{(\lambda_{i}^{+})^2} + \sum_{i\in\inds} \frac{1}{(\lambda_{i}^-)^2} & = \frac{2}{\sinh^2(\period/2)} \biggl(\int_a^{a+\period} u(x)dx\biggr)^2 + 4\coth(\period/2) \int_a^{a+\period} d\mu. 
 \end{align} 
Apart from this, the second trace formula~\eqref{eqnTF2} immediately provides a lower estimate for the moduli of the periodic/antiperiodic eigenvalues. 
      
   \begin{corollary}\label{corLB2}
    For every $\lambda\in\sigma_\pm$ we have the bound
\begin{align}\label{eqnLB2}
\frac{1}{\lambda^2} \leq \frac{\pm 1}{\cosh(\period/2) \mp 1} \biggl(\int_a^{a+\period} u(x)dx\biggr)^2 
    + \frac{2\sinh(\period/2)}{\cosh(\period/2) \mp 1} \int_a^{a+\period} d\mu.
\end{align}
   \end{corollary}

 Under the additional restriction that the periodic/antiperiodic spectrum is positive, the first trace formula~\eqref{eqnTF1} yields another lower estimate that coincides with the one found in \cite[Theorem 4.1]{chmezh18} using a different approach.
      
   \begin{corollary}\label{cor:minlam}
    Suppose that the periodic/antiperiodic spectrum is positive. 
    Then for every $\lambda\in\sigma_\pm$ we have the bound 
 \begin{align}
     \frac{1}{\lambda} & \leq  \frac{\sinh(\period/2)}{\cosh(\period/2) \mp 1} \int_a^{a+\period} u(x)dx. 
 \end{align}
   \end{corollary}
   
   One should note that the estimates  in Corollary~\ref{corLB2} and Corollary~\ref{cor:minlam} are sharp.
  In fact, equality holds for some $\lambda\in\sigma_\pm$ if and only if the periodic/antiperiodic spectrum consists precisely of one eigenvalue.
 That this case indeed occurs can be seen from \cite[Theorem~5.4]{InvPeriodMP} for example.

  In a similar manner as before, we will now exploit the product representation\footnote{We employ the convention that a fraction with $\pm\infty$ in the denominator is regarded to be zero.}
 \begin{align}\label{eqnF}
  s(z,a+\period) = 2 \sinh(\period/2) \cprod_{i\in\inds} \biggl( 1-\frac{z}{\kappa_i}\biggr), \quad z\in\C,
 \end{align}
 to derive trace formulas for the Dirichlet spectrum. 
 
 \begin{proposition}\label{propTID}
  We have the identities 
   \begin{align}
       \label{eqnTFD1}\csum_{i\in\inds} \frac{1}{\kappa_{i}} & =  \coth(\period/2)\int_a^{a+\period} u(x) dx -2u(a), \\
    \label{eqnTFD2}\sum_{i\in\inds} \frac{1}{\kappa_{i}^2} & =  \frac{1}{\sinh^2(\period/2)} \biggl(\int_a^{a+\period} u(x)dx\biggr)^2 + 2\coth(\period/2) \int_a^{a+\period} d\mu - 8P(a), 
\end{align}
where $P$ is the $\period$-periodic function given by 
  \begin{align}
  P(x) = \frac{1}{4}  \int_\R \E^{-|x-\xi|} u(\xi)^2 d\xi + \frac{1}{4} \int_\R \E^{-|x-\xi|} d\mu(\xi), \quad x\in\R.
 \end{align}
 \end{proposition}

\begin{proof}
 The first identity follows from the first formula in~\eqref{eqncprodatzero} applied to the product~\eqref{eqnF} in conjunction with~\eqref{eqnds0}.  
For the second identity, we observe that integration by parts turns~\eqref{eqndds0} into 
 \begin{align*}
    \ddot{s}(0,a+\period) & = 2\sinh(\period/2)\biggl(\int_a^{a+\period} u(x)dx\biggr)^2 - 8 \cosh(\period/2) u(a)\int_a^{a+\period} u(x)dx \\ 
     & \qquad - 4\cosh(\period/2) \int_a^{a+\period} d\mu  + 8 \sinh(\period/2)u(a)^2 + 16 \sinh(\period/2)P(a),
 \end{align*}
 upon noting that the function $P$ is a solution of the differential equation
  \begin{align*}
  P - P'' =\frac{u^2+\mu}{2}. 
 \end{align*} 
 Now the claim readily follows from the second formula in~\eqref{eqncprodatzero} applied to the product~\eqref{eqnF}.
 \end{proof} 

 Let us point out that by comparing the trace formulas for the periodic/anti\-periodic eigenvalues in~\eqref{eqnTF1c} and~\eqref{eqnTF2c} with the trace formulas for the Dirichlet eigenvalues in~\eqref{eqnTFD1} and~\eqref{eqnTFD2}, we see that it is always possible to explicitly recover the quantities $u(a)$ and $P(a)$ from these three spectra.

 \section{Continuous dependence}\label{secCE}

  The goal of this section is to prove that the periodic/antiperiodic eigenvalues as well as the Dirichlet eigenvalues depend continuously on the pair $(u,\mu)$ with respect to a weak$^\ast$ topology.
 More precisely, we endow the set $\Dper$ with the initial topology with respect to the functionals 
\begin{align}\label{eqnFctweakstar}
 (u,\mu) & \mapsto \int_\R u(x) h(x)dx + \int_\R u'(x) h'(x)dx + \int_\R g\, d\mu
\end{align}
 for all functions $h\in H^1_{\cc}(\R)$ and $g\in C_\cc(\R)$.
 In order to state our results, let $(u_k,\mu_k)$ be a sequence of pairs in $\Dper$ and denote all associated quantities in the same way as for $(u,\mu)$ in the previous sections but with an additional subscript $k$. 
 
\begin{theorem} \label{thContM}
 If the pairs $(u_k,\mu_k)$ converge to $(u,\mu)$ in $\Dper$, then the Floquet discriminants $\Delta_k$ converge locally uniformly to $\Delta$.
\end{theorem} 
 
\begin{proof}
Suppose that the sequence $(u_k, \mu_k)$ converges to $(u,\mu)$ in $\Dper$. 
Upon choosing the test function  
\begin{align*}
  h(x) = \begin{cases} \E^{2\period - |x-x_0|} - \E^{|x-x_0|}, & x\in[x_0-\period,x_0+\period], \\ 0, & x\not\in[x_0-\period,x_0+\period], \end{cases}
\end{align*}
in~\eqref{eqnFctweakstar}, we see that $u_k$ converges to $u$ in every point $x_0\in\R$. 
Moreover, if $I=[c_1,c_2]$ is a compact interval and $h_I$ is a function in $H^1(I)$, then we have 
\begin{align*}
 & \int_I f(x)h_I(x)dx +  \int_I f'(x)h_I'(x)dx  \\
 & \qquad = \int_\R f(x) h(x)dx + \int_\R f'(x) h'(x)dx -\tanh(\period/2) (f(c_1)h(c_1)+f(c_2)h(c_2))
\end{align*}
for every $f\in H^1_\per(\R)$, where $h\in H^1_\cc(\R)$ is defined by 
\begin{align*}
  h(x)  = \begin{cases}
    h_I(c_1)\sinh(\period)^{-1}\sinh(x-c_1+\period), & x\in[c_1-\period,c_1), \\
    h_I(x), & x\in [c_1,c_2], \\
    h_I(c_2)\sinh(\period)^{-1}\sinh(c_2+\period-x), & x\in(c_2,c_2+\period], \\
    0, & x\not\in[c_1-\period,c_2+\period].
  \end{cases}
\end{align*}
It now follows readily that $u_k|_I$ converges to $u|_I$ weakly in $H^1(I)$, from which we conclude that $u_k$ converges to $u$ weakly in $H^1_\loc(\R)$. 
As this entails locally uniform convergence, we see that  
\begin{align*}
    \int_\R  u_k'(x)^2 g(x) dx +  \int_\R g\, d\dip_k  = \int_\R g \, d \mu_k - \int_\R u_k(x)^2 g(x) dx 
  \end{align*}
  converges to 
  \begin{align*}
     \int_\R g \, d \mu - \int_\R u(x)^2 g(x) dx =  \int_\R u'(x)^2 g(x) dx +  \int_\R g\, d\dip
  \end{align*}
  for all functions $g\in C_\cc(\R)$.

Since the Floquet discriminants are independent of the chosen base point $a$, we can assume that the measure $\dip$ has no mass in $a$. 
By Lemma~\ref{lemAppIVPsolcon}, the Floquet discriminants $\Delta_k$ then converge pointwise to $\Delta$. 
 In order to show that the convergence is in fact locally uniform, first note that by Proposition~\ref{propcsent} and Lemma~\ref{lemDelta}, the functions $\Delta_k$ are real entire of Cartwright class with only real and simple roots. 
 Therefore, they belong to the P\'olya class (see \cite[Problem 9]{dB68}) and we see from~\eqref{eqndDel0} and~\eqref{eqnddDel0} that the sequences $\dot{\Delta}_k(0)$ and $\ddot{\Delta}_k(0)$ are bounded. 
 It now remains to employ a compactness argument using~\cite[Lemma 4.4]{lawo}. 
\end{proof} 

As the periodic/antiperiodic eigenvalues are precisely the zeros of $\Delta\mp1$, it now follows from Hurwitz's theorem (see \cite[Theorem~VII.2.5]{con} for example) that the periodic/antiperiodic spectrum depends continuously on the pair $(u,\mu)$. 
That the same is true for the Dirichlet spectrum as well can be seen from our last result. 

\begin{theorem}\label{thContD}
 If the pairs $(u_k,\mu_k)$ converge to $(u,\mu)$ in $\Dper$, then the entire functions $s_k(\ledot,a+\period)$ converge locally uniformly to $s(\ledot,a+\period)$. 
\end{theorem}

\begin{proof}
 Pick a base point $\tilde{a}\in\R$ such that the measure $\dip$ has no mass in $\tilde{a}$ and consider the corresponding fundamental system of solutions $\tilde{c}_k(z,\redot)$, $\tilde{s}_k(z,\redot)$ of the differential equation~\eqref{eqnDEhok} with the initial conditions 
  \begin{align*}
   \tilde{c}_k(z,\tilde{a}) & = \tilde{s}_k^\qd(z,\tilde{a}) = 1, & \tilde{c}_k^\qd(z,\tilde{a}) &  = \tilde{s}_k(z,\tilde{a}) = 0, 
 \end{align*} 
 for every $z\in\C$.
Clearly, we have
\begin{align*}
s_k(z,x) =  \tilde{c}_k(z,a)  \tilde{s}_k(z,x) -\tilde{s}_k(z,a) \tilde{c}_k(z,x), \quad x\in\R,~z\in\C,
\end{align*}
so that we can apply Lemma~\ref{lemAppIVPsolcon} to the right-hand side, which shows that $s_k(\ledot,a+\period)$ converges to $s(\ledot,a+\period)$ pointwise. 
Now the claim follows after another compactness argument as in the proof of Theorem~\ref{thContM} using~\eqref{eqnds0} and~\eqref{eqndds0}.
\end{proof}


\end{document}